\documentclass[11pt]{article}

\usepackage[a4paper,margin=27mm]{geometry}
\usepackage{amsmath,amssymb,amsthm,mathtools}
\usepackage{microtype}
\usepackage{enumitem}
\usepackage{booktabs}
\usepackage[hidelinks]{hyperref}

\newtheorem{theorem}{Theorem}[section]
\newtheorem{proposition}[theorem]{Proposition}
\newtheorem{lemma}[theorem]{Lemma}
\newtheorem{corollary}[theorem]{Corollary}
\newtheorem{remark}[theorem]{Remark}
\newtheorem{definition}[theorem]{Definition}

\newcommand{\F}{\mathbb F}
\newcommand{\Z}{\mathbb Z}
\newcommand{\sltwo}{\mathfrak{sl}_2}
\newcommand{\Ad}{\operatorname{Ad}}
\newcommand{\ord}{\operatorname{ord}}
\newcommand{\tr}{\operatorname{tr}}
\newcommand{\rank}{\operatorname{rank}}
\newcommand{\Cent}{\operatorname{Cent}}
\newcommand{\im}{\operatorname{Im}}
\newcommand{\PRE}{\mathrm{PRE}}
\newcommand{\ACT}{\mathrm{ACT}}

\title{\textbf{Adjoint Transfer and Order-Lifting Dichotomies\\
for \(2\times2\) Matrices over Prime-Power Rings}}
\author{Ivo Vajc\\
\small Independent researcher\\
\small ORCID: 0009-0005-5145-6491}
\date{11 October 2026}

\begin{document}
\maketitle

\begin{abstract}
For an odd prime $p$, we study the change in matrix order under
lifting from $\mathbb Z/p^k\mathbb Z$ to $\mathbb Z/p^{k+1}\mathbb Z$.
The first variation of the order defect is an affine translate of a
cyclic adjoint transfer on $\mathfrak{sl}_2(\mathbb F_p)$.
We classify the full-order transfer for every $q\in GL_2(\mathbb F_p)$:
its ranks are $3$ for scalar elements, $1$ for nonscalar semisimple
elements, $1$ for nonscalar Jordan elements when $p=3$, and $0$ for
nonscalar Jordan elements when $p>3$.
For a Jordan element $q=\lambda(I+U)$, $U^2=0\ne U$, the vanishing
first-order transfer for $p>3$ admits an exact higher-order refinement.
Writing $d=(p-1)/2$, the full-order perturbation has depth $d$ when
$\operatorname{tr}(UX)\ne0$, depth $p-1$ when
$\operatorname{tr}(UX)=0$ and $X\notin\langle U\rangle$, and vanishes
identically for $X\in\langle U\rangle$.
For $p=3$, among the $27$ lifts of a nonscalar order-preserving parent
with determinant fixed to $+1$ or $-1$, exactly $9$ preserve its order
and $18$ triple it; after the first increase, every further lift
triples the order. We also give a scalar obstruction criterion for
order preservation among unrestricted Jordan lifts.
\end{abstract}



\section{Introduction}

Let \(p\) be an odd prime.  A basic lifting problem asks how the order of
a matrix modulo \(p^k\) changes when the matrix is lifted to modulus
\(p^{k+1}\).  The purpose of this paper is to organize this question
through a tangent-space description of determinant-compatible lifts.

The relevant tangent space is \(\sltwo(\F_p)\).  The first-order change
of the order defect is not a bare linear invariant: it is an
\emph{affine} translate of a cyclic adjoint transfer.  This distinction
is essential for obtaining exact fiber counts.

The central phenomenon is especially rigid in characteristic \(3\).
For every non-scalar reduction the relevant full-order transfer has rank one.
Consequently a three-dimensional tangent space splits into affine
fibers of size \(3^2=9\), producing the exact count
\[
27=9+18.
\]
After the first order increase, a periodic-transfer cancellation and a
prime-power propagation identity make the activated state absorbing.

The vanishing of the Jordan transfer for \(p>3\) is only a first-order
statement.  A higher-order analysis developed below shows that the dormant
Jordan signal reappears at the exact generic perturbative depth
\[
 d=\frac{p-1}{2}.
\]
At that degree the response is confined to the nilpotent centralizer
axis and has one-dimensional image span.  For \(p>3\) this critical
response is generally nonlinear as a function of \(X\), so no linear-rank
claim is intended.  Its amplitude is controlled by the quadratic character
of the transverse tangent coordinate.  This identifies characteristic \(3\) not
as an isolated anomaly, but as the depth-one endpoint of a uniform
odd-characteristic law.

The two principal results are the global transfer classification and the
exact higher-order Jordan depth law, together with the characteristic-three
fixed-determinant lifting theorem. Their proofs are given below.

\subsection*{Relation to previous work}
The congruence tangent space and its traceless Lie-algebra structure
are classical \cite{LeeSzczarba1975,Lopez2014}. Modular cyclic norms
and unipotent Jordan modules are standard tools in representation theory
\cite{Alperin1986,Benson1998,McNinch2002}; related adjoint-orbit
problems over finite quotient rings are studied in \cite{Zordan2016}.
Fibonacci periods and companion-matrix methods provide historical
motivation, rather than a novelty claim
\cite{Wall1960,Miles1960,BragmanRowland2025}.
The results below concern the specific affine order-defect formula,
its exact lifting counts, and the higher-order Jordan perturbation
formula. The references are not an exhaustive priority search.

\section{Congruence lifts and order defects}

Fix an odd prime \(p\).  Throughout we use the convention
\[
 \Ad_g(X):=gXg^{-1}.
\]
Thus \(\Ad_{q^{-j}}(X)=q^{-j}Xq^j\), consistently with the transfer sums below.
Let
\[
G_{p^k}=\{L\in GL_2(\Z/p^k\Z):\det L\in\{+1,-1\}\}.
\]
We consider compatible lifts from \(G_{p^k}\) to \(G_{p^{k+1}}\).

Here $\det L\in\{+1,-1\}$ is an equality of residue classes in $\Z/p^k\Z$; it does not require a chosen integer representative of $L$ to have integer determinant exactly $+1$ or $-1$.

Let \(q\in GL_2(\F_p)\) be non-scalar.  Define
\[
 A_q=q-\frac{\tr q}{2}I,
 \qquad
 \Delta(q)=(\tr q)^2-4\det q.
\]
Cayley--Hamilton gives
\begin{equation}
 A_q^2=\frac{\Delta(q)}4 I.
 \label{eq:A-square}
\end{equation}

For an integer \(r\ge1\), define the length-\(r\) adjoint sum
\[
 N_q^{(r)}(X)=\sum_{j=0}^{r-1}q^{-j}Xq^j,
 \qquad X\in\sltwo(\F_p).
\]
When \(\rho(q)=\ord(\Ad_q)\), we call
\[
 \mathcal N_q=N_q^{(\rho(q))}
\]
the adjoint transfer over one adjoint period.

\begin{lemma}[Matrix order versus adjoint period]
Let
\[
 r_0=\ord(q),\qquad \rho=\ord(\Ad_q).
\]
Then \(\rho\mid r_0\), \(q^\rho\) is scalar, and
\begin{equation}
 N_q^{(r_0)}
 =
 \frac{r_0}{\rho}\,\mathcal N_q
 \quad\text{in }\operatorname{End}_{\F_p}(\sltwo(\F_p)).
 \label{eq:full-vs-adjoint}
\end{equation}
More generally, for every multiple \(r=s\rho\),
\[
 N_q^{(r)}=s\,\mathcal N_q.
\]
\end{lemma}

\begin{proof}
The equality \(\Ad_q^\rho=1\) means that \(q^\rho\) commutes with every
traceless matrix. Since $p$ is odd, every $2\times2$ matrix is a sum
of a traceless matrix and a scalar matrix (subtract half its trace).
Thus $q^\rho$ commutes with the entire matrix algebra. Its center
consists precisely of scalar matrices, so \(q^\rho\) is scalar.  Since \(q^{r_0}=I\), the adjoint
period divides \(r_0\).  The length-\(r_0\) sum decomposes into
\(r_0/\rho\) consecutive blocks of length \(\rho\), and every block is
the same because \(\Ad_q^\rho=1\).  This gives
\eqref{eq:full-vs-adjoint}.
\end{proof}

\begin{corollary}[No period-collapse over \(\F_p\)]
Let \(p\) be prime, \(q\in GL_2(\F_p)\), and
\[
 r_0=\ord(q),\qquad \rho=\ord(\Ad_q).
\]
Then
\[
 s:=\frac{r_0}{\rho}
\]
is the multiplicative order of the scalar \(q^\rho\in\F_p^\times I\).
Consequently
\[
 s\mid(p-1),
 \qquad
 p\nmid s,
\]
and
\[
 N_q^{(r_0)}=s\,\mathcal N_q
\]
has the same kernel, image, and rank as \(\mathcal N_q\).
\end{corollary}

\begin{proof}
By the preceding lemma, \(q^\rho=cI\) for some
\(c\in\F_p^\times\), and \(\rho\mid r_0\).  Write \(r_0=\rho s\).
Then
\[
 c^sI=q^{\rho s}=q^{r_0}=I.
\]
If \(c^t=1\) for some \(0<t<s\), then
\(q^{\rho t}=I\), contradicting the minimality of \(r_0\).
Hence \(s=\ord(c)\).  Since \(c\in\F_p^\times\), its order divides
\(p-1\), so \(s\) is nonzero in \(\F_p\).  Equation
\eqref{eq:full-vs-adjoint} therefore differs from \(\mathcal N_q\)
only by multiplication by a nonzero field scalar, which preserves
kernel, image, and rank.
\end{proof}

\begin{remark}
The distinction between the matrix order \(r_0\) and the adjoint
period \(\rho\) remains notationally and structurally important, but
over \(GL_2(\F_p)\) it cannot produce a rank collapse: the repetition
factor \(r_0/\rho\) always lies in \(\F_p^\times\).
For \(p=3\) this specializes to \(r_0/\rho\in\{1,2\}\).
\end{remark}

\subsection{Determinant-compatible tangent lifts}

\begin{lemma}[Tangent torsor]
Fix a determinant-compatible reference lift \(C_0\) from modulus
\(p^k\) to \(p^{k+1}\).  Every determinant-compatible lift over the same
parent can be written uniquely in the form
\[
 C=C_0(I+p^kX)\pmod{p^{k+1}},
 \qquad X\in\sltwo(\F_p).
\]
Hence there are \(p^3\) such tangent lifts.
\end{lemma}

\begin{proof}
Modulo \(p^{k+1}\), a relative correction is \(I+p^kX\).
The first-order determinant formula gives
\[
 \det(I+p^kX)\equiv1+p^k\tr X\pmod{p^{k+1}}.
\]
Determinant compatibility is therefore equivalent to
\(\tr X=0\) in \(\F_p\).  Since \(\dim_{\F_p}\sltwo=3\), there are
\(p^3\) possibilities.
\end{proof}

\subsection{Order defect and its first variation}

Let \(m=p^k\), let \(C\) be a lift, and suppose the parent order is \(r\).
Whenever
\[
 C^r\equiv I\pmod m,
\]
define the order defect
\[
 D_r(C)=\frac{C^r-I}{m}\pmod p.
\]

\begin{proposition}[Exact one-step order-defect dichotomy]
Let \(m=p^k\), let the reduction of \(C\) modulo \(m\) have exact
order \(r\), and suppose
\[
 C^r\equiv I+mD\pmod{pm}.
\]
Then the order of \(C\) modulo \(pm\) is either \(r\) or \(pr\), and
\[
 D=0\quad\Longleftrightarrow\quad \ord_{pm}(C)=r,
\]
while
\[
 D\ne0\quad\Longleftrightarrow\quad \ord_{pm}(C)=pr.
\]
\end{proposition}

\begin{proof}
The order modulo \(pm\) is a multiple of the exact order \(r\) modulo
\(m\).  On the other hand,
\[
 (I+mD)^p\equiv I\pmod{pm},
\]
because every nonconstant term is divisible by \(pm=p^{k+1}\).
Hence \(C^{pr}\equiv I\pmod{pm}\), so the only possibilities are
\(r\) and \(pr\).  The first occurs exactly when \(C^r\equiv I\pmod{pm}\),
equivalently when \(D=0\).
\end{proof}

\begin{proposition}[Affine first variation]
Let
\[
 C=C_0(I+mX)\pmod{pm},
 \qquad X\in\sltwo(\F_p),
\]
and let \(q=C_0\bmod p\).  Then
\begin{equation}
 D_r(C)-D_r(C_0)=N_q^{(r)}(X).
 \label{eq:first-variation}
\end{equation}
Thus the defect family is the affine map
\[
 X\longmapsto D_0+N_q^{(r)}(X),
 \qquad D_0=D_r(C_0).
\]
\end{proposition}

\begin{proof}
Expand the product
\[
[C_0(I+mX)]^r
\]
to first order in \(m\).  Terms containing two factors \(mX\) vanish
modulo \(pm\).  Moving the single tangent insertion through successive
powers of \(C_0\) yields the conjugation sum in
\eqref{eq:first-variation}.
\end{proof}

\section{Classification of the adjoint transfer}
\subsection{The centralizer axis}

\begin{lemma}
For every non-scalar \(q\in GL_2(\F_p)\),
\[
 \Cent(q)\cap\sltwo(\F_p)=\langle A_q\rangle.
\]
\end{lemma}

\begin{proof}
The centralizer of a non-scalar \(2\times2\) matrix is the
two-dimensional algebra \(\F_p[I,q]\).  Its traceless subspace is
one-dimensional, and \(A_q\) is a nonzero traceless element of it.
\end{proof}

\begin{proposition}[Image-axis principle]
For every complete adjoint period,
\[
 \im\mathcal N_q\subseteq\langle A_q\rangle.
\]
Consequently
\[
 \rank\mathcal N_q\le1.
\]
\end{proposition}

\begin{proof}
Cyclically shifting the transfer sum shows
\[
 q^{-1}\mathcal N_q(X)q=\mathcal N_q(X).
\]
The transfer also preserves tracelessness.  Its image therefore lies
in \(\Cent(q)\cap\sltwo(\F_p)=\langle A_q\rangle\).
\end{proof}

\subsection{Transfer classification and explicit operator formulas}

We now sharpen the rank classification to an explicit classification of the
full-order operator
\[
 N_q:=N_q^{(r_0)},\qquad r_0=\ord(q).
\]
The earlier adjoint-period transfer \(\mathcal N_q\) remains useful for the
order-lifting argument, while the formulas below identify the full matrix-order
transfer itself.

\subsubsection{Nonscalar semisimple case}

Assume \(\Delta(q)\ne0\), and put
\[
 Z_q=A_q=q-\frac{\tr(q)}2I.
\]
By \eqref{eq:A-square},
\[
 \tr(Z_q^2)=\frac{\Delta(q)}2\ne0.
\]

\begin{lemma}[Semisimple fixed line]
If \(q\) is nonscalar semisimple, then
\[
 \Cent(q)\cap\sltwo(\F_p)=\langle Z_q\rangle,
\]
and \(p\nmid r_0\).
\end{lemma}

\begin{proof}
The centralizer statement was proved above.  In the split case the
eigenvalues of \(q\) lie in \(\F_p^\times\); in the irreducible case they
lie in \(\F_{p^2}^\times\).  Both multiplicative groups have order prime to
\(p\), so the order \(r_0\) of a semisimple element is prime to \(p\).
\end{proof}

\begin{theorem}[Semisimple centralizer-projection formula]
Let \(q\) be nonscalar semisimple and \(r_0=\ord(q)\).  Define
\[
 P_q(X)=\frac{\tr(Z_qX)}{\tr(Z_q^2)}Z_q.
\]
Then
\[
 \boxed{N_q=r_0P_q.}
\]
Consequently
\[
 \im N_q=\Cent(q)\cap\sltwo(\F_p)=\langle Z_q\rangle,
 \qquad
 \ker N_q=Z_q^\perp,
 \qquad
 \rank N_q=1.
\]
\end{theorem}

\begin{proof}
Because \(q^{r_0}=I\), cyclic reindexing gives
\[
 \Ad_qN_q=N_q=N_q\Ad_q.
\]
Thus \(\im N_q\) lies in the fixed space of \(\Ad_q\), namely
\(\langle Z_q\rangle\).  Since \(p\nmid r_0\), the normalized average
\(E_q=r_0^{-1}N_q\) is the projector onto that fixed space.  The trace
pairing \(\langle X,Y\rangle=\tr(XY)\) is nondegenerate on
\(\sltwo(\F_p)\) for odd \(p\), is preserved by conjugation, and
\(\tr(Z_q^2)\ne0\).  Hence the normalized cyclic average is precisely the
trace-orthogonal projector onto \(\langle Z_q\rangle\), which is \(P_q\).
Therefore \(N_q=r_0P_q\).
\end{proof}

\subsubsection{Modular cyclic norm viewpoint}

Let \(C_p=\langle g\rangle\) act on a vector space over characteristic
\(p\).  Its cyclic norm element is
\[
 \mathbf N=1+g+\cdots+g^{p-1}.
\]
Since \(g^p-1=(g-1)^p\), one has the classical group-algebra identity
\begin{equation}
 \boxed{\mathbf N=(g-1)^{p-1}.}
 \label{eq:cyclic-norm-power}
\end{equation}
The classification of indecomposable modules for cyclic \(p\)-groups in
characteristic \(p\) is classical \cite{Alperin1986,Benson1998}.  In the
Jordan adjoint sector below, \((g-1)^3=0\), so \eqref{eq:cyclic-norm-power}
already forces the \(p\)-step norm to vanish for \(p>3\).  The explicit
calculation identifies what survives at \(p=3\).

\subsubsection{Nonscalar Jordan case}

Assume \(\Delta(q)=0\).  Put
\[
 \lambda=\frac{\tr(q)}2.
\]
Since \(q\) is nonscalar and invertible,
\[
 q=\lambda(I+U),\qquad U\ne0,\quad U^2=0,
\]
and if \(m=\ord(\lambda)\), then \(\ord(q)=mp\).
Scalar multiplication does not affect conjugation.  For \(j\in\F_p\),
\[
 (I+U)^j=I+jU,
 \qquad
 (I+U)^{-j}=I-jU,
\]
so
\begin{equation}
 \Ad_{(I+U)^j}(X)=X+j[U,X]-j^2UXU.
 \label{eq:jordan-polynomial}
\end{equation}

\begin{lemma}[Finite-field quadratic moment]
For odd prime \(p\),
\[
 \sum_{j\in\F_p}1=0,
 \qquad
 \sum_{j\in\F_p}j=0,
\]
and
\[
 \sum_{j\in\F_p}j^2=
 \begin{cases}
 -1,&p=3,\\
 0,&p>3.
 \end{cases}
\]
\end{lemma}

\begin{proof}
The first two identities are standard.  Over \(\F_p^\times\), the sum of
\(k\)-th powers vanishes unless \(p-1\mid k\), in which case it equals
\(-1\).  For \(k=2\) and odd \(p\), this occurs exactly at \(p=3\).
\end{proof}

\begin{lemma}[Rank of the exceptional channel]
For nonzero nilpotent \(U\in M_2(\F_p)\) with \(U^2=0\), the map
\[
 J_U:\sltwo(\F_p)\to\sltwo(\F_p),\qquad X\mapsto UXU,
\]
has rank one.
\end{lemma}

\begin{proof}
After conjugation and nonzero scalar rescaling, \(U\) may be replaced by
\(E_{12}\).  For
\[
 X=\begin{pmatrix}a&b\\c&-a\end{pmatrix}
\]
one obtains \(E_{12}XE_{12}=cE_{12}\).  Hence the image is exactly one
dimensional.
\end{proof}

\begin{theorem}[Jordan full-order transfer dichotomy]
Let \(q=\lambda(I+U)\) be nonscalar Jordan, with \(m=\ord(\lambda)\).
Then
\[
 \boxed{
 N_q(X)=
 \begin{cases}
 m\,UXU,&p=3,\\
 0,&p>3.
 \end{cases}}
\]
Consequently
\[
 \rank N_q=
 \begin{cases}
 1,&p=3,\\
 0,&p>3.
 \end{cases}
\]
\end{theorem}

\begin{proof}
Summing \eqref{eq:jordan-polynomial} over one unipotent \(p\)-cycle and
using the moment lemma gives \(X\mapsto UXU\) at \(p=3\) and the zero map
for \(p>3\).  Since \(\ord(q)=mp\), the full matrix-order transfer consists
of \(m\) identical unipotent cycles.  Hence it is \(m(X\mapsto UXU)\) for
\(p=3\) and zero for \(p>3\).  Since \(m\mid p-1\), its coefficient is
nonzero modulo \(p\); the rank-one statement follows from the preceding
lemma.
\end{proof}

\begin{remark}[Why characteristic \(3\) is exceptional]
The Jordan conjugation orbit is quadratic in the cyclic parameter.  Its
constant and linear moments vanish for every odd prime.  The quadratic
moment vanishes for \(p>3\), but at \(p=3\) its degree equals \(p-1\), so it
survives.  Thus characteristic \(3\) is the critical odd characteristic
for this \(2\times2\) adjoint problem.
\end{remark}

\subsubsection{Unified global transfer theorem}

The scalar case is also transparent: if \(q=\lambda I\) and
\(r_0=\ord(\lambda)\), then \(\Ad_q=I\) and
\[
 N_q=r_0I.
\]
Since \(r_0\mid p-1\), this has rank three.

\begin{theorem}[Global Adjoint Transfer Classification]
\label{thm:global-transfer}
Let \(p\) be an odd prime and \(q\in GL_2(\F_p)\).  Put
\[
 N_q=\sum_{j=0}^{\ord(q)-1}\Ad_{q^{-j}}
 \quad\text{on }\sltwo(\F_p).
\]
Then exactly one of the following occurs:
\[
 \boxed{
 N_q=
 \begin{cases}
 r_0I,&q=\lambda I,\quad r_0=\ord(\lambda),\\[1mm]
 r_0P_q,&q\text{ nonscalar semisimple},\\[1mm]
 m(X\mapsto UXU),&q=\lambda(I+U)\text{ Jordan},\ p=3,\\[1mm]
 0,&q=\lambda(I+U)\text{ Jordan},\ p>3.
 \end{cases}}
\]
Accordingly,
\[
 \boxed{
 \rank N_q=
 \begin{cases}
 3,&q\text{ scalar},\\
 1,&q\text{ nonscalar semisimple},\\
 1,&q\text{ nonscalar Jordan},\ p=3,\\
 0,&q\text{ nonscalar Jordan},\ p>3.
 \end{cases}}
\]
\end{theorem}

\begin{proof}
The scalar formula follows from trivial conjugation.  The nonscalar
semisimple formula is the centralizer-projection theorem.  The repeated-root
case is the Jordan full-order transfer dichotomy.  These cases are mutually
exclusive and exhaustive for \(2\times2\) matrices over a field.
\end{proof}

\begin{corollary}[Global kernel--image dimensions]
Under Theorem~\ref{thm:global-transfer},
\[
 (\dim\ker N_q,\dim\im N_q)=
 \begin{cases}
 (0,3),&q\text{ scalar},\\
 (2,1),&q\text{ nonscalar semisimple},\\
 (2,1),&q\text{ Jordan},\ p=3,\\
 (3,0),&q\text{ Jordan},\ p>3.
 \end{cases}
\]
\end{corollary}

\begin{proof}
The ambient space $\sltwo(\F_p)$ has dimension three.  Applying
rank--nullity to the four cases of Theorem~\ref{thm:global-transfer}
gives the stated kernel and image dimensions.
\end{proof}

\subsubsection{Response--transfer corollary}

Let \(B_1,\dots,B_s\in M_2(\F_p)\) be trace probes whose restrictions
span \(\sltwo(\F_p)^*\), and define
\[
 R_q(X)=\bigl(\tr(B_1N_q(X)),\dots,\tr(B_sN_q(X))\bigr).
\]

\begin{corollary}[Response--transfer duality]
For every \(q\in GL_2(\F_p)\),
\[
 \boxed{\ker R_q=\ker N_q,\qquad \rank R_q=\rank N_q.}
\]
Thus the same \(3/1/1/0\) classification is the exact first-order
observational reconstruction classification for a nondegenerate trace-probe
family.
\end{corollary}

\begin{proof}
The trace pairing on \(\sltwo(\F_p)\) is nondegenerate for odd \(p\), and
the assumed probes span its dual.  Hence the measurement map is injective on
\(\sltwo(\F_p)\).  Therefore all responses vanish exactly when
\(N_q(X)=0\), giving equality of kernels and, by rank--nullity, equality of
ranks.
\end{proof}

\begin{remark}[Dormancy versus measurement blindness]
If \(N_q(X)\ne0\) but a restricted probe family annihilates it, the failure
is measurement blindness.  If \(N_q(X)=0\), the first-order signal has
already vanished before measurement.  The nonscalar Jordan sector for odd
\(p>3\) is therefore an exact transfer-dormant sector.
\end{remark}

\section{Higher-order Jordan perturbations}
\label{sec:higher-jordan}

The first-order classification above leaves a precise question in the
nonscalar Jordan sector for \(p>3\): does the vanishing transfer represent
absolute loss of perturbative information, or only delayed visibility?
The following theorem gives an exact answer.

Let
\[
 q=\lambda(I+U),\qquad U^2=0,\quad U\ne0,
\]
with \(\lambda\in\F_p^\times\), and let
\[
 A(\varepsilon)=q(I+\varepsilon X),\qquad
 X\in\sltwo(\F_p).
\]
Write
\[
 m=\ord_{\F_p^\times}(\lambda),\qquad
 d=\frac{p-1}{2}.
\]
Since \(I+U\) has order \(p\) and \(m\mid p-1\), the matrix \(q\) has
order \(mp\).

\begin{lemma}[Traceless Cayley--Hamilton decomposition]
\label{lem:ch-discriminant}
Let \(R\) be a commutative ring in which \(2\) is invertible, and let
\(A\in M_2(R)\).  Put
\[
 t=\tr(A),\qquad Z=A-\frac t2I,\qquad
 \Delta(A)=t^2-4\det(A).
\]
Then
\[
 Z^2=\frac{\Delta(A)}4I.
\]
If moreover \(R\) is an \(\F_p\)-algebra with \(p\) odd, then
\[
 \boxed{
 A^p=\left(\frac t2\right)^pI+
 \left(\frac{\Delta(A)}4\right)^{(p-1)/2}Z.}
\]
\end{lemma}

\begin{proof}
The matrix \(Z\) is traceless, so Cayley--Hamilton gives
\(Z^2=-\det(Z)I\).  Directly,
\(-\det(Z)=\Delta(A)/4\).  For the second assertion, the scalar matrix
\((t/2)I\) commutes with \(Z\), and the intermediate binomial coefficients
vanish in characteristic \(p\).  Hence
\[
 A^p=\left(\frac t2I+Z\right)^p
 =\left(\frac t2\right)^pI+Z^p.
\]
Writing \(p=2d+1\) and using \(Z^2=(\Delta/4)I\) gives
\(Z^p=(\Delta/4)^dZ\).
\end{proof}

\begin{lemma}[Jordan discriminant expansion]
\label{lem:jordan-discriminant-expansion}
For \(A(\varepsilon)=q(I+\varepsilon X)\) as above, put
\(c=\tr(UX)\).  Then the discriminant is exactly
\[
 \boxed{
 \Delta(A(\varepsilon))
 =\lambda^2\bigl(4c\varepsilon+(c^2-4\det X)\varepsilon^2\bigr).}
\]
In particular its linear term is \(4\lambda^2c\varepsilon\), and
\[
 A(0)-\frac{\tr A(0)}2I=\lambda U.
\]
\end{lemma}

\begin{proof}
Since \(U\) and \(X\) are traceless,
\[
 \tr A(\varepsilon)
 =\lambda\bigl(2+\varepsilon\tr(UX)\bigr).
\]
Moreover \(\det(I+U)=1\), while for a traceless \(2\times2\) matrix \(X\),
\[
 \det(I+\varepsilon X)=1+\varepsilon^2\det X.
\]
Substitution into \(\Delta=(\tr A)^2-4\det A\) gives the displayed
exact quadratic formula.  At \(\varepsilon=0\), \(\tr(q)=2\lambda\), so the traceless part
of \(q\) is \(\lambda U\).
\end{proof}

\begin{theorem}[Jordan Perturbative Depth Theorem]
\label{thm:jordan-depth}
Let \(p\) be an odd prime,
\[
 q=\lambda(I+U)\in GL_2(\F_p),\qquad U^2=0,\quad U\ne0,
\]
and let \(X\in\sltwo(\F_p)\).  Put
\[
 m=\ord_{\F_p^\times}(\lambda),\qquad d=\frac{p-1}{2}.
\]
Then in \(M_2(\F_p[[\varepsilon]])\),
\[
 \boxed{
 [q(I+\varepsilon X)]^{mp}
 =I+m\bigl(\tr(UX)\bigr)^dU\,\varepsilon^d
 +O(\varepsilon^{d+1}).}
\]
Equivalently,
\[
 [\varepsilon^j][q(I+\varepsilon X)]^{mp}=0
 \qquad(1\le j<d),
\]
while
\[
 \boxed{
 [\varepsilon^d][q(I+\varepsilon X)]^{mp}
 =m\bigl(\tr(UX)\bigr)^dU.}
\]
\end{theorem}

\begin{proof}
Apply Lemma~\ref{lem:ch-discriminant} to \(A=A(\varepsilon)\):
\[
 A^p=\left(\frac{\tr A}{2}\right)^pI+
 \left(\frac{\Delta(A)}4\right)^d
 \left(A-\frac{\tr A}{2}I\right).
\]
The Frobenius map sends a power series \(\sum a_j\varepsilon^j\) to
\(\sum a_j^p\varepsilon^{jp}\).  Since \(d<p\), the scalar term has no
positive-degree contribution through degree \(d\).  By
Lemma~\ref{lem:jordan-discriminant-expansion},
\[
 \left(\frac{\Delta(A)}4\right)^d
 =\lambda^{2d}\bigl(\tr(UX)\bigr)^d\varepsilon^d
 +O(\varepsilon^{d+1}).
\]
Now \(2d=p-1\), so \(\lambda^{2d}=1\), and the constant term of the
traceless factor is \(\lambda U\).  Therefore
\[
 A^p=\lambda I+
 \lambda\bigl(\tr(UX)\bigr)^dU\,\varepsilon^d
 +O(\varepsilon^{d+1}).
\]
Factoring out \(\lambda\) and raising to the \(m\)-th power gives
\[
 A^{mp}=I+m\bigl(\tr(UX)\bigr)^dU\,\varepsilon^d
 +O(\varepsilon^{d+1}),
\]
because \(\lambda^m=1\) and every product of two positive-degree terms
has degree at least \(2d\ge d+1\).
\end{proof}

\begin{corollary}[Generic Jordan perturbative depth]
\label{cor:jordan-depth}
If \(\tr(UX)\ne0\), then the first nonzero full-order perturbative
coefficient occurs exactly at
\[
 \boxed{d_{\min}(p)=\frac{p-1}{2}.}
\]
If \(\tr(UX)=0\), Theorem~\ref{thm:jordan-depth} gives vanishing through
degree \(d\), but does not by itself determine the first later nonzero
degree.
\end{corollary}

\begin{proof}
If $\tr(UX)\ne0$, then its $d$-th power is nonzero in the field.
Theorem~\ref{thm:jordan-depth} shows that every lower positive-degree
coefficient vanishes and that the coefficient at degree $d$ is nonzero.
When $\tr(UX)=0$, that same theorem only implies vanishing through
degree $d$, as asserted.
\end{proof}

\begin{definition}[Full-order perturbative depth]
\label{def:full-order-depth}
For the present Jordan problem define
\[
 R_X(\varepsilon)=[q(I+\varepsilon X)]^{mp}-I.
\]
The full-order perturbative depth is
\[
 D_R(X)=\min\{j\ge1:[\varepsilon^j]R_X(\varepsilon)\ne0\},
\]
with \(D_R(X)=\infty\) if \(R_X(\varepsilon)\equiv0\) as a polynomial
matrix over \(\F_p\).  Thus infinite depth below refers only to this
full-order perturbative response.
\end{definition}

\begin{theorem}[Three-Layer Jordan Reconstruction-Depth Theorem]
\label{thm:jordan-three-layer}
Under the hypotheses of Theorem~\ref{thm:jordan-depth}, let
\[
 H_U=\{X\in\sltwo(\F_p):\tr(UX)=0\}.
\]
Then
\[
 \boxed{
 D_R(X)=
 \begin{cases}
 d,&X\notin H_U,\\
 p-1,&X\in H_U\setminus\langle U\rangle,\\
 \infty,&X\in\langle U\rangle,
 \end{cases}
 \qquad d=\frac{p-1}{2}.}
\]
For the second layer the leading coefficient has the invariant form
\[
 \boxed{
 [\varepsilon^{p-1}][q(I+\varepsilon X)]^{mp}
 =m(-\det X)^dU,
 \qquad X\in H_U\setminus\langle U\rangle.}
\]
In fact \((-\det X)^d=1\) on this stratum, so the coefficient is \(mU\).
On the final axis one has the exact identity
\[
 \boxed{[q(I+\varepsilon X)]^{mp}=I
 \qquad(X\in\langle U\rangle).}
\]
\end{theorem}

\begin{proof}
The first stratum is Corollary~\ref{cor:jordan-depth}.  It remains to
analyze \(H_U\).  Simultaneous conjugation of \(q,U,X\) preserves
\(\tr(UX)\), \(\det X\), membership in \(\langle U\rangle\), and the
full-order response up to conjugation, so we may take
\[
 U=\begin{pmatrix}0&1\\0&0\end{pmatrix}.
\]
Writing a general traceless matrix as
\[
 X=\begin{pmatrix}a&b\\c&-a\end{pmatrix},
\]
one has \(\tr(UX)=c\).  Hence on \(H_U\),
\[
 X=\begin{pmatrix}a&b\\0&-a\end{pmatrix},
 \qquad \det X=-a^2.
\]
Moreover \(X\in\langle U\rangle\) exactly when \(a=0\).

Suppose first that \(X\in H_U\setminus\langle U\rangle\), so \(a\ne0\).
The exact discriminant formula of
Lemma~\ref{lem:jordan-discriminant-expansion} becomes
\[
 \Delta(A(\varepsilon))=4\lambda^2a^2\varepsilon^2.
\]
Thus, with \(d=(p-1)/2\),
\[
 \left(\frac{\Delta}{4}\right)^d
 =a^{p-1}\varepsilon^{p-1}.
\]
Here \(\lambda^{p-1}=1\).  Also \(\tr A=2\lambda\), so the scalar
term in Lemma~\ref{lem:ch-discriminant} is \(\lambda I\).  Since
\[
 A-\lambda I=\lambda\bigl(U+\varepsilon(X+UX)\bigr),
\]
we obtain the exact polynomial identity
\[
 A^p=\lambda\left[
 I+a^{p-1}\varepsilon^{p-1}U
 +a^{p-1}\varepsilon^p(X+UX)\right].
\]
Raising to the \(m\)-th power shows that the coefficient of
\(\varepsilon^{p-1}\) in \(A^{mp}\) is \(ma^{p-1}U=mU\), and no
smaller positive degree occurs.  Because \(-\det X=a^2\), this is
precisely
\[
 m(-\det X)^dU.
\]
This proves the second stratum and its invariant coefficient formula.

Finally, if \(X=bU\), then
\[
 q(I+\varepsilon X)
 =\lambda\bigl[I+(1+b\varepsilon)U\bigr].
\]
Since \(U^2=0\), in characteristic \(p\)
\[
 [I+sU]^p=I+psU=I.
\]
Therefore \(A^p=\lambda I\) and \(A^{mp}=I\) identically in
\(\varepsilon\).  Hence \(D_R(X)=\infty\) on the final axis.
\end{proof}

\begin{corollary}[Exact Jordan response flag]
\label{cor:jordan-response-flag}
In the setting of Theorem~\ref{thm:jordan-three-layer}, the invariant flag
\[
 \sltwo(\F_p)\supset H_U\supset\langle U\rangle
\]
has dimensions \(3,2,1\), respectively.  Consequently its three response
strata have cardinalities
\[
 \boxed{
 |\sltwo(\F_p)\setminus H_U|=p^2(p-1),\qquad
 |H_U\setminus\langle U\rangle|=p(p-1),\qquad
 |\langle U\rangle|=p.}
\]
The corresponding full-order perturbative depths are
\[
 \boxed{\frac{p-1}{2},\qquad p-1,\qquad \infty.}
\]
In particular, the second finite response depth is exactly twice the first.
\end{corollary}

\begin{proof}
The linear functional \(X\mapsto\tr(UX)\) is nonzero, so \(H_U\) is a
codimension-one subspace of the three-dimensional space
\(\sltwo(\F_p)\).  Hence \(|H_U|=p^2\).  Since \(U\ne0\), the axis
\(\langle U\rangle\) is one-dimensional and has \(p\) elements.  Taking
successive set differences gives the displayed cardinalities, while the
depths are those of Theorem~\ref{thm:jordan-three-layer}.
\end{proof}

\begin{corollary}[Jordan discriminant-contact law]
\label{cor:jordan-discriminant-contact}
In the setting of Theorem~\ref{thm:jordan-three-layer}, with the
convention \(v_\varepsilon(0)=\infty\),
\[
 v_\varepsilon(\Delta_X)=
 \begin{cases}
 1,&X\notin H_U,\\
 2,&X\in H_U\setminus\langle U\rangle,\\
 \infty,&X\in\langle U\rangle,
 \end{cases}
\]
and consequently
\[
 \boxed{D_R(X)=\frac{p-1}{2}\,v_\varepsilon(\Delta_X).}
\]
This identity is asserted only for the present \(2\times2\) Jordan
perturbation problem; no general discriminant-depth principle is claimed.
\end{corollary}

\begin{proof}
The exact expansion in Lemma~\ref{lem:jordan-discriminant-expansion}
has a nonzero linear term when $\tr(UX)\ne0$.  On $H_U$, its
quadratic coefficient equals $-4\lambda^2\det X$.  In the coordinates
used in Theorem~\ref{thm:jordan-three-layer}, this is nonzero precisely
when $X\notin\langle U\rangle$.  On $\langle U\rangle$ the entire
discriminant vanishes.  Comparing these three valuations with the
three depths of Theorem~\ref{thm:jordan-three-layer} gives the formula.
\end{proof}

\begin{corollary}[Critical Jordan axis and quadratic character]
\label{cor:jordan-character}
At the critical degree, the response lies in
\[
 \langle U\rangle=C_{\sltwo(\F_p)}(q).
\]
For \(\tr(UX)\ne0\), Euler's criterion gives
\[
 \boxed{
 [\varepsilon^d][q(I+\varepsilon X)]^{mp}
 =m\left(\frac{\tr(UX)}p\right)U,}
\]
where \((\frac{\cdot}{p})\) denotes the Legendre symbol.
\end{corollary}

\begin{proof}
Theorem~\ref{thm:jordan-depth} places the critical coefficient in
$\langle U\rangle$.  For $c=\tr(UX)\ne0$, Euler's criterion in
$\F_p$ gives $c^{(p-1)/2}=(\frac{c}{p})$.  Substitute this identity
into the coefficient formula of that theorem.
\end{proof}

\begin{corollary}[Recovery of the characteristic-three Jordan law]
For \(p=3\), one has \(d=1\), and
\[
 [\varepsilon][q(I+\varepsilon X)]^{3m}
 =m\tr(UX)U.
\]
For a nonzero nilpotent \(2\times2\) matrix \(U\) and traceless \(X\),
\(UXU=\tr(UX)U\).  Hence
\[
 [\varepsilon][q(I+\varepsilon X)]^{3m}=mUXU,
\]
which is exactly the full-order characteristic-three Jordan transfer law
of Theorem~\ref{thm:global-transfer}.
\end{corollary}

\begin{proof}
At $p=3$ one has $d=1$.  Conjugating $U$ to $E_{12}$ and writing
$X=\left(\begin{smallmatrix}a&b\\c&-a\end{smallmatrix}\right)$
gives $UXU=cU=\tr(UX)U$.  The first-order coefficient of
Theorem~\ref{thm:jordan-depth} is therefore $mUXU$, as required.
\end{proof}

\begin{remark}[First-order dormancy as delayed visibility]
Theorem~\ref{thm:jordan-depth} does not alter the first-order transfer
classification: for nonscalar Jordan elements, \(N_q=0\) when \(p>3\).
It explains that vanishing as a depth effect.  Characteristic \(3\) is the
unique odd characteristic for which \((p-1)/2=1\); for larger odd primes
the same Jordan channel first becomes visible at higher perturbative order.
\end{remark}

\section{Order lifting in characteristic three}
\subsection{Affine fiber geometry}

Suppose the relevant length-\(r\) norm in
\eqref{eq:first-variation} has rank one and the affine equation
\[
 D_0+N_q^{(r)}(X)=0
\]
is solvable.  Then every zero fiber is a coset of the kernel.

Since
\[
 \dim\sltwo(\F_p)=3,\qquad\rank N=1,
\]
we have
\[
 \dim\ker N=2,\qquad |\ker N|=p^2.
\]

\begin{corollary}[Rank-one branching]
In a rank-one solvable affine state, exactly \(p^2\) of the \(p^3\)
determinant-compatible tangent lifts have zero defect and
\[
p^3-p^2=p^2(p-1)
\]
have nonzero defect.
\end{corollary}

\begin{proof}
If the affine equation has a solution $X_0$, its solution set is
$X_0+\ker N$.  Rank--nullity gives $\dim\ker N=2$, so there are
$p^2$ zero-defect solutions among the $p^3$ tangent lifts.  All
remaining $p^3-p^2$ lifts have nonzero defect.
\end{proof}

Thus the local decomposition is
\[
 \boxed{p^3=p^2+p^2(p-1)}.
\]

\subsection{Order-preserving states and persistence}
Here $\PRE$ (order-preserving) means exact order $r_0=\ord(q)$;
$\ACT$ (activated) means that the order has acquired a positive
power of $3$ relative to $r_0$.

We now specialize to \(p=3\).  Let \(q\in GL_2(\F_3)\) be non-scalar and
write
\[
 r_0=\ord(q).
\]
By the transfer classification and the general no-period-collapse result, the
full-order transfer \(N_q^{(r_0)}\) has rank one and
\[
 \im N_q^{(r_0)}=\langle A_q\rangle.
\]

\begin{lemma}[Defect-axis lemma]
Let \(C_k\in G_{3^k}\) have exact order \(r_0\), and let \(C_0\) be any
determinant-compatible lift to modulus \(3^{k+1}\).  Write
\[
 C_0^{r_0}=I+3^kD\pmod{3^{k+1}},\qquad D\in M_2(\F_3).
\]
Then
\[
 [q,D]=0,\qquad \tr D=0,
\]
and hence
\[
 D\in \Cent(q)\cap\sltwo(\F_3)=\langle A_q\rangle.
\]
\end{lemma}

\begin{proof}
Since \(C_0\) commutes with its power \(C_0^{r_0}\),
\[
 0=[C_0,C_0^{r_0}]
   \equiv 3^k[C_0,D]\pmod{3^{k+1}}.
\]
Dividing by \(3^k\) and reducing modulo \(3\) gives \([q,D]=0\).

Next, determinant compatibility gives \(\det C_0=\varepsilon\) with
\(\varepsilon\in\{\pm1\}\) at modulus \(3^{k+1}\).  Because the
reduction of \(C_0\) modulo \(3^k\) has order \(r_0\), its determinant
satisfies \(\varepsilon^{r_0}=1\); in particular, if
\(\varepsilon=-1\), then \(r_0\) is even.  Thus
\[
 1=\det(C_0^{r_0})
  =\det(I+3^kD)
  \equiv1+3^k\tr D\pmod{3^{k+1}},
\]
so \(\tr D=0\) in \(\F_3\).  For non-scalar \(q\), its centralizer in
\(M_2(\F_3)\) is \(\langle I,q\rangle\); intersecting with the traceless
subspace leaves the one-dimensional axis \(\langle A_q\rangle\).
\end{proof}

\begin{theorem}[PRE-Persistence Theorem]
Let \(q\in GL_2(\F_3)\) be non-scalar and let
\(r_0=\ord(q)\).  Suppose \(C_k\in G_{3^k}\) reduces to \(q\) modulo
\(3\) and has exact order \(r_0\).  Then among the \(27\)
determinant-compatible lifts of \(C_k\) to modulus \(3^{k+1}\), exactly
\(9\) have exact order \(r_0\) and exactly \(18\) have exact order
\(3r_0\).  In particular, every PRE state has PRE descendants at the
next level, so PRE persistence holds for all \(k\ge1\).
\end{theorem}

\begin{proof}
Choose any determinant-compatible reference lift \(C_0\).  By the
defect-axis lemma,
\[
 D_{r_0}(C_0)\in\langle A_q\rangle.
\]
In characteristic three the full-order transfer has rank one and
\[
 \im N_q^{(r_0)}=\langle A_q\rangle.
\]
Consequently
\[
 -D_{r_0}(C_0)\in\im N_q^{(r_0)},
\]
so the affine correction equation
\[
 D_{r_0}(C_0)+N_q^{(r_0)}(X)=0
\]
is automatically solvable.  The determinant-compatible lifts form an
affine torsor over the three-dimensional space \(\sltwo(\F_3)\).
Since \(N_q^{(r_0)}\) has rank one, its kernel has dimension two and
contains exactly \(3^2=9\) elements.  Hence the zero-defect affine
fiber has exactly \(9\) elements.  By the exact one-step order-defect
dichotomy, these \(9\) lifts have exact order \(r_0\), while the other
\(18\) have exact order \(3r_0\).

Every one of the nine order-preserving children again has exact order
\(r_0\), the same non-scalar reduction \(q\), and determinant
\(\pm1\).  Therefore the same argument applies at the next level and,
by induction, at every subsequent level.
\end{proof}

\begin{corollary}[PRE step with fixed determinant]
For every non-scalar characteristic-three PRE state,
\[
 \boxed{\PRE\longrightarrow9\PRE+18\ACT.}
\]
\end{corollary}

\begin{proof}
This is the transition count asserted by the PRE-Persistence Theorem:
its nine order-preserving children are PRE and its eighteen children
of order $3r_0$ are newly activated.
\end{proof}

\subsection{Unrestricted Jordan lifts: a scalar obstruction criterion}
\label{sec:unrestricted-jordan}

The determinant condition in the preceding PRE-persistence theorem is
essential to that theorem's automatic solvability conclusion.  This section treats \emph{all} lifts in $GL_2$, without fixing
the determinant at the next level.  The resulting obstruction is not
automatically zero.  In particular, the statements below must not be
confused with the $27$-element determinant-compatible torsor.

\begin{lemma}[Cyclic-basis divisibility]
\label{lem:cyclic-divisibility}
Let $R=\mathbb Z/3^k\mathbb Z$ and let $A\in M_2(R)$ have nonscalar
reduction $q$ modulo $3$.  There is a vector $w\in R^2$ such that
$(w,Aw)$ is an $R$-basis.  Consequently, if
$uA+vI=I$ in $M_2(R)$, then $u=0$ and $v=1$ in $R$.
\end{lemma}
\begin{proof}
A nonscalar $2\times2$ matrix over a field has a cyclic vector:
otherwise every one-dimensional subspace would be invariant, forcing
it to be scalar.  Choose $\bar w$ with $(\bar w,q\bar w)$ linearly
independent over $\mathbb F_3$, and lift it to $w\in R^2$.
The determinant of $(w,Aw)$ is a unit in $R$, hence this pair is a
basis.  Applying $(uA+(v-1)I)$ to $w$ gives
$uAw+(v-1)w=0$; the basis property forces $u=v-1=0$.
\end{proof}

\begin{lemma}[Jordan full-matrix norm]
\label{lem:jordan-full-norm}
Let $q=\lambda I+N\in GL_2(\mathbb F_3)$, with
$N\ne0$, $N^2=0$, and $\lambda\in\mathbb F_3^\times$.
For $n=\ord(q)\in\{3,6\}$ define
\[
 \mathcal N_{q,n}(X)=\sum_{j=0}^{n-1}q^{-j}Xq^j,
 \qquad X\in M_2(\mathbb F_3).
\]
Then
\[
 \mathcal N_{q,n}(X)=\lambda^{n-1}\operatorname{tr}(NX)N.
\]
In particular its image is $\langle N\rangle$, its rank on the
four-dimensional full matrix space is one, and the trace map on its
kernel is surjective.
\end{lemma}
\begin{proof}
Write $q=\lambda(I+U)$ with $U=\lambda^{-1}N$, so $U^2=0$.
Then $(I+U)^j=I+jU$ and $(I+U)^{-j}=I-jU$ in characteristic three.
The conjugation sum for $n=3$ is
\[
 \sum_{j=0}^2(I-jU)X(I+jU)
 =\sum_{j=0}^2\bigl(X+j(XU-UX)-j^2UXU\bigr)
 =UXU,
\]
since $\sum_{j=0}^2 j=0$ and $-\sum_{j=0}^2j^2=1$ in
$\mathbb F_3$.  For a rank-one nilpotent $U$, the identity
$UXU=\operatorname{tr}(UX)U$ follows by choosing a Jordan basis.
For $n=6$, the conjugation sum is twice the $n=3$ sum, and the
coefficient becomes $\lambda^{5}$ in the displayed formula
($\lambda=1$ when $n=3$, $\lambda=-1$ when $n=6$).
Finally $I\in\ker\mathcal N_{q,n}$ and $\operatorname{tr}(I)=2\ne0$
in $\mathbb F_3$, proving trace-surjectivity.
\end{proof}

\begin{theorem}[Unrestricted Jordan order-preserving lifts]
\label{thm:unrestricted-jordan}
Fix $k\ge1$, put $m=3^k$, and let
$A\in GL_2(\mathbb Z/m\mathbb Z)$ reduce to a nonscalar Jordan
matrix $q=\lambda I+N$ over $\mathbb F_3$.
Assume that $A$ has exact order
$n=\ord(q)\in\{3,6\}$ (a \emph{PRE} parent).
Choose any integral representative $A_0$ of $A$, and set
$t=\operatorname{tr}(A_0)$, $d=\det(A_0)$.
Define integral polynomials $u_j(t,d),v_j(t,d)$ by
\[
 (u_0,v_0)=(0,1),\qquad
 (u_1,v_1)=(1,0),\qquad
 u_{j+1}=t u_j+v_j,\quad v_{j+1}=-d u_j.
\]
Thus $A_0^n=u_n A_0+v_nI$ by Cayley--Hamilton.
By Lemma~\ref{lem:cyclic-divisibility}, $m\mid u_n$ and
$m\mid(v_n-1)$.  Define
\[
 \alpha=(u_n/m)\bmod3,\qquad
 \beta=((v_n-1)/m)\bmod3,\qquad
 s=\lambda\alpha+\beta\in\mathbb F_3.
\]
Then the following are equivalent:
\begin{enumerate}[label=(\roman*)]
\item $s=0$;
\item the order defect of $A_0$ lies in the image of the full
      first-variation operator;
\item there exists a lift $B\in GL_2(\mathbb Z/3^{k+1}\mathbb Z)$
      of $A$ with exact order $n$.
\end{enumerate}
If these conditions hold, there are exactly $27$ order-preserving
lifts among the $81$ lifts of $A$, with exactly $9$ in each of the
three possible determinant lifts.  If $s\ne0$, none of the $81$
lifts preserves order.  The condition $s=0$ is independent of the
chosen integral representative $A_0$.
\end{theorem}
\begin{proof}
Since $A_0^n\equiv I\pmod m$, the Cayley--Hamilton recurrence and
Lemma~\ref{lem:cyclic-divisibility} give the stated divisibilities.
Modulo $3m$ the order defect is
\[
 E=\frac{A_0^n-I}{m}\bmod3
   =\alpha q+\beta I
   =\alpha N+sI.
\]
Every lift of $A$ can be written uniquely as
$B=A_0(I+mX)\pmod{3m}$, with $X\in M_2(\mathbb F_3)$.
First-order expansion gives
\[
 \frac{B^n-I}{m}\equiv E+\mathcal N_{q,n}(X)\pmod3.
\]
By Lemma~\ref{lem:jordan-full-norm}, the image of the norm is
$\langle N\rangle$.  As $I$ and $N$ are linearly independent,
$E$ belongs to that image if and only if $s=0$.
In that case the solution set is a coset of the norm kernel,
which has $3^{4-1}=27$ elements.  The trace functional is
surjective on the kernel, so each of its three trace fibers has
$27/3=9$ solutions.  Since
\[
 \det B\equiv\det A_0(1+m\operatorname{tr}X)\pmod{3m},
\]
these are exactly the three determinant fibers.  The one-step
order dichotomy identifies zero defect with exact order $n$.
Changing $A_0$ changes $E$ by an element of the norm image, hence
cannot change whether $s$ vanishes; equivalently the scalar
component of $E$ modulo $\langle N\rangle$ is intrinsic.
\end{proof}

\begin{corollary}[Reconciliation with determinant-compatible persistence]
\label{cor:reconcile}
If, in addition, $\det A=\varepsilon\in\{+1,-1\}$ modulo $3^k$,
then the $27$ lifts with determinant exactly $\varepsilon$ modulo
$3^{k+1}$ include exactly $9$ order-preserving lifts and $18$
order-increasing lifts.
\end{corollary}
\begin{proof}
Choose a reference lift $A_0$ with determinant exactly
$\varepsilon$ modulo $3m$.  As $n=3$ or $6$ and
$\varepsilon^n=1$, we have
$\det(A_0^n)=1\pmod{3m}$, so $\operatorname{tr}E=0$.
But $E=\alpha N+sI$ has trace $2s$ in $\mathbb F_3$.
Thus $s=0$, and the theorem gives nine solutions in the
prescribed determinant fiber.  The other eighteen lifts in that
fiber increase the order by three.
\end{proof}

\begin{remark}[Why the two enumerations are different]
For the $16$ nonscalar Jordan roots modulo $3$, there are $432$
order-preserving first-level lifts if all three determinant fibers
are admitted.  At modulus $9$, an exhaustive audit classifies these
$432$ parents as $144$ with $s=0$ and $288$ with $s\ne0$;
their unrestricted next-level order-preserving counts are respectively
$27$ and $0$.  This does \emph{not} contradict the determinant-constrained
PRE theorem: the obstructed parents have determinant outside the
fixed $\pm1$ classes.  The finite counts corroborate, but are not
used in the proof of, Theorem~\ref{thm:unrestricted-jordan}.
The observed $9$-versus-$18$ division of the $27$ unrestricted
order-preserving children into next-level unobstructed and obstructed
parents is a separate finite observation, not asserted here as an
all-level branching theorem.
\end{remark}

\subsection{Absorbing activation}

\begin{lemma}[Characteristic-three propagation]
\label{lem:cubic-propagation}
Let $k\ge1$ and let $E$ be a matrix over the integers (or over a
compatible $3$-adic matrix ring). Then
\[
 (I+3^kE)^3\equiv I+3^{k+1}E\pmod{3^{k+2}}.
\]
\end{lemma}

\begin{proof}
The binomial expansion is valid because $I$ commutes with $E$:
\[
 (I+3^kE)^3=I+3^{k+1}E+3^{2k+1}E^2+3^{3k}E^3.
\]
For $k\ge1$ both $2k+1\ge k+2$ and $3k\ge k+2$.
\end{proof}

\begin{lemma}[First-activation boundary and universal propagation]
\label{lem:first-act-boundary}
Let $k\ge1$, and let $B$ be a matrix modulo $3^{k+1}$ whose reduction
modulo $3^k$ has exact order $r$. Suppose that $B$ has exact order
$3r$ modulo $3^{k+1}$. Then there is a uniquely determined nonzero
matrix $D\in M_2(\F_3)$ such that
\[
 B^r\equiv I+3^kD\pmod{3^{k+1}},\qquad D\ne0.
\]
For \emph{every} lift $\widetilde B$ of $B$ modulo $3^{k+2}$,
regardless of determinant restrictions, one has
\[
 \boxed{\widetilde B^{3r}\equiv
 I+3^{k+1}D\pmod{3^{k+2}}.}
\]
In particular, every such lift has exact order $9r$ modulo $3^{k+2}$.
More generally, for any $n\ge0$ and any lift $\widehat B_n$ of $B$
modulo $3^{k+2+n}$,
\[
 \boxed{\widehat B_n^{3^{n+1}r}\equiv
 I+3^{k+1+n}D\pmod{3^{k+2+n}}.}
\]
Consequently its reduction modulo $3^{k+1+n}$ has exact order
$3^{n+1}r$.
\end{lemma}

\begin{proof}
Since $B$ reduces to a matrix of order $r$, the one-step order-defect
dichotomy implies that $B^r\equiv I+3^kD\pmod{3^{k+1}}$ for a unique
$D\in M_2(\F_3)$, and that $D\ne0$ because the order of $B$ is $3r$.
For any next lift $\widetilde B$ there is a matrix $E$ such that
\[
 \widetilde B^r\equiv I+3^kD+3^{k+1}E\pmod{3^{k+2}}.
\]
Apply Lemma~\ref{lem:cubic-propagation} to the matrix
$D+3E$ to obtain
\[
 \widetilde B^{3r}=(\widetilde B^r)^3
 \equiv I+3^{k+1}(D+3E)
 \equiv I+3^{k+1}D\pmod{3^{k+2}}.
\]
The right-hand defect is nonzero, hence the exact one-step dichotomy,
applied to the parent order $3r$, forces exact order $9r$.
At each successive level the same argument applies with $k$ increased
by one and the current order multiplied by three. The leading matrix
$D\pmod3$ is unchanged. Induction proves the asserted formula and
exact orders. No enumeration, rank assumption, or determinant condition
is used in this propagation step.
\end{proof}

\begin{proposition}[Absorbing activated step]
\label{prop:absorbing-act}
Let $q\in GL_2(\F_3)$ be nonscalar, with $r_0=\ord(q)$.
Consider any determinant-compatible lifting tower rooted at $q$.
If at some step the exact order first increases from $r_0$ to $3r_0$,
then every determinant-compatible descendant at every later level
has its order multiplied by three at each subsequent step.
In particular every activated parent has exactly $27$ activated
children, and the transition $\ACT\to27\ACT$ is absorbing.
\end{proposition}

\begin{proof}
The first order increase supplies a nonzero defect by the exact
one-step dichotomy. Lemma~\ref{lem:first-act-boundary} applies to
\emph{every} next lift and recursively to every subsequent lift,
showing that the current order triples at each level. The number
of determinant-compatible children of a parent is $3^3=27$ by the
tangent-torsor lemma. Thus all $27$ children are activated and the
argument repeats indefinitely.
\end{proof}

\begin{remark}[Transfer cancellation is consistent but not required]
If the current order is $R=3^a r_0$, $a\ge1$, then
\[
 N_q^{(R)}=3^aN_q^{(r_0)}=0\quad\text{over }\F_3.
\]
Thus all children have the same first-order $R$-defect. The boundary
lemma proves more: that common defect is necessarily nonzero after
first activation, and gives its exact value at every later level.
This removes the nonzero-defect assumption from the earlier
conditional activated-step formulation.
\end{remark}

\begin{theorem}[Recursive branching with fixed determinant]
\label{thm:unconditional-branching}
Consider a characteristic-three lifting tower with nonscalar reduction
and determinant fixed to $+1$ or $-1$ at every level. Let $\PRE$
denote exact order $r_0$, and $\ACT$ an order increased by a positive
power of three. Then
\[
 \boxed{\PRE\longrightarrow9\PRE+18\ACT,\qquad
        \ACT\longrightarrow27\ACT.}
\]
\end{theorem}

\begin{proof}
The PRE-Persistence Theorem supplies exactly nine order-preserving
and eighteen first-activated children of each PRE parent.
Proposition~\ref{prop:absorbing-act} shows that each such activated
child and all its descendants have strictly increasing powers of three
in their exact orders. The tangent-torsor lemma gives exactly $27$
children per parent. These are the two claimed transitions.
\end{proof}

\section{Consequences and discussion}
\subsection{Counting consequences}

Let \(P_n\) and \(A_n\) denote the numbers of PRE and ACT descendants
after \(n\) further lifting levels.  Then
\[
 \begin{pmatrix}P_{n+1}\\A_{n+1}\end{pmatrix}
 =
 \begin{pmatrix}9&0\\18&27\end{pmatrix}
 \begin{pmatrix}P_n\\A_n\end{pmatrix}.
\]
After normalization by \(27\),
\[
 \begin{pmatrix}1/3&0\\2/3&1\end{pmatrix}
\]
is the transition matrix for the corresponding proportions.

Starting from a PRE root, the surviving PRE proportion after \(d\)
levels is
\[
 3^{-d}.
\]
If each of the $27$ determinant-compatible children is selected with equal
probability, independently at each level, the first-activation depth
therefore has distribution
\[
 \Pr(D=d)=\frac{2}{3^d},\qquad d=1,2,\ldots,
\]
and
\[
 \mathbb E[D]=\frac32.
\]

\subsection{Discussion and scope}
The order defect varies affinely by the cyclic adjoint transfer. This
connects a linear-algebraic image and kernel computation to exact counts
of order-preserving and order-increasing lifts. In characteristic three,
the nonscalar fixed-determinant case yields the branching law of
Theorem~\ref{thm:unconditional-branching}; it is not asserted for scalar
reductions, arbitrary odd primes, or unrestricted determinant fibers.

For $p>3$, the nonscalar Jordan transfer vanishes at first order,
but the formal perturbative response has the three depths of
Theorem~\ref{thm:jordan-three-layer}. Infinite depth means exact
vanishing of this particular formal matrix response, not an assertion
about physical observability. These formulas use classical congruence,
Jordan, and modular-transfer ingredients; historical priority of their
particular combination is not established here.

\subsection{Computational checks}
Finite computations were used to test special cases and detect errors,
not as premises of the proofs. Checks included the transfer ranks,
first-variation identity, characteristic-three lift counts, cubic
defect propagation, and Jordan perturbative coefficients at selected
small odd primes. Reproducibility requires publication of the exact
scripts, test parameters, and outputs; this manuscript does not claim
that such a companion archive is currently available.

\appendix
\section{Finite verification: scope}
The algebraic statements in the main text are proved without relying
on enumeration. Auxiliary finite checks covered the characteristic-three
nonscalar matrices, determinant-compatible lifts at small levels, and
selected Jordan perturbations for $p=3,5,7,11,13$. These checks are
limited to the tested parameters and are not proofs for all primes or
all congruence levels. No particular public code archive, checksum
manifest, or permanent identifier is asserted in this version.

\section*{Declaration of AI assistance}
OpenAI ChatGPT was used for mathematical exploration, drafting, and
programming assistance. The author is responsible for the submitted
claims, proofs, references, and text. Computational checks and AI
assistance do not replace independent mathematical verification.

\end{document}